\documentclass[graybox]{svmult}

\usepackage{mathptmx}
\usepackage{helvet}
\usepackage{courier}

\usepackage{makeidx}
\usepackage{multicol}
\usepackage[bottom]{footmisc}

\usepackage{amssymb}
\usepackage[misc,geometry]{ifsym}

\usepackage{amsmath}

\makeindex

\begin{document}

\title*{Complex-valued fractional derivatives\\
on time scales\thanks{This is a preprint of a paper 
whose final and definite form will appear in
\emph{Springer Proceedings in Mathematics \& Statistics},
ISSN: 2194-1009. Accepted for publication 06/Nov/2015.}}


\author{Benaoumeur Bayour and Delfim F. M. Torres}

\authorrunning{B. Bayour and D. F. M. Torres}


\institute{Benaoumeur Bayour
\at University of Chlef, B.P. 151, Hay Es-salem Chlef, Algeria,
\email{b.benaoumeur@gmail.com}
\and
Delfim F. M. Torres \Letter
\at University of Aveiro,
Center for Research and Development in Mathematics and Applications (CIDMA),
Department of Mathematics, 3810-193 Aveiro, Portugal,
\email{delfim@ua.pt}}


\maketitle


\abstract{We introduce a notion of fractional (noninteger order)
derivative on an arbitrary nonempty closed subset of the real numbers
(on a time scale). Main properties of the new operator are proved
and several illustrative examples given.}


\ \



\section{Introduction}

The study of fractional (noninteger) order derivatives
on discrete, continuous and, more generally,
on an arbitrary nonempty closed set (i.e., a time scale)
is a well-known subject under strong current development.
The subject is very rich and several different definitions
and approaches are available, either in discrete \cite{MR3362702},
continuous \cite{MR3353441}, and time-scale settings \cite{NBD}.
In continuous time, i.e., for the time scale $\mathbb{T}=\mathbb{R}$,
several definitions are based on the classical Euler Gamma function $\Gamma$.
For the time scale $\mathbb{T}=\mathbb{Z}$, the Gamma function is nothing
else than the factorial, while for the $q$-scale one has
the $q$-Gamma function $\Gamma_q$ \cite{MR1865777}.
For the definition of Gamma function on an arbitrary time scale $\mathbb{T}$
see \cite{MR3136530}. Similarly to \cite{NBD,MyID:320},
here we introduce a new notion of fractional derivative
on an arbitrary time scale $\mathbb{T}$ that does not
involve Gamma functions. Our approach is, however, different from the ones
available in the literature \cite{NBD,MyID:320,MyID:328,MyID:324}.
In particular, while in \cite{NBD,MyID:320,MyID:328,MyID:324}
the fractional derivative at a point is always a real number,
here, in contrast, the fractional derivative at a point is,
in general, a complex number. For example, the derivative
of order $\alpha \in (0, 1]$ of the square function $t^2$
is always given by $t^\alpha + (\sigma(t))^\alpha$, where
$\sigma(t)$ is the forward jump operator of the time scale,
which is in general a complex number (e.g., for $\alpha = 1/2$ and $t < 0$)
and a generalization of the Hilger derivative $(t^2)^\Delta = t + \sigma(t)$.

The text is organized as follows. In Section~\ref{sec:2} we recall
the notion of Hilger/delta derivative. Our complex-valued fractional
derivative on time scales is introduced in Section~\ref{sec:3},
where its main properties are proved and several examples given.


\section{Preliminaries}
\label{sec:2}

A time scale $\mathbb{T}$ is an arbitrary nonempty
closed subset of the real numbers $\mathbb{R}$.
For $t\in\mathbb{T}$, we define the forward jump operator
$\sigma:\mathbb{T}\rightarrow\mathbb{T}$
by $\sigma(t)=\inf\{s\in\mathbb{T}:s>t\}$
and the backward jump operator $\rho:\mathbb{T}\rightarrow\mathbb{T}$
is defined by $\rho(t):=\sup\{s\in\mathbb{T}:s<t\}$.
Then, one defines the graininess function
$\mu:\mathbb{T}\rightarrow[0,+\infty[$
by $\mu(t)= \sigma(t)-t$.

If $\sigma(t)>t$, then we say that $t$ is right-scattered;
if $\rho(t)<t $, then $t$ is left-scattered. Moreover,
if $t < \sup\mathbb{T}$ and $\sigma(t)=t$, then $t$ is called right-dense;
if $t>\inf\mathbb{T}$ and $\rho(t)=t$, then $t$ is called left-dense.
If $\mathbb{T}$ has a left-scattered maximum $m$, then we define
$\mathbb{T}^{\kappa}=\mathbb{T}\setminus\{ m\}$; otherwise
$\mathbb{T}^{\kappa}=\mathbb{T}$.
If $f:\mathbb{T}\rightarrow \mathbb{R}$, then
$f^{\sigma}:\mathbb{T}\rightarrow \mathbb{R}$ is given by
$f^{\sigma}(t)=f(\sigma(t))$ for all $t\in\mathbb{T}$.

\begin{definition}[The Hilger derivative \cite{MBP}]
Let $f:\mathbb{T}\rightarrow \mathbb{R}$ and $t\in\mathbb{T}$.
We define $f^{\Delta}(t)$ to be the number
(provided it exists) with the property that given any
$\epsilon>0$ there is a neighborhood $U$ of $t$
(i.e., $U=(t-\delta,t+\delta)\cap\mathbb{T}$ for some $\delta>0$) such that
$$
|[f(\sigma(t))-f(s)]-f^{\Delta}(t)[\sigma(t)-s]|\leq\epsilon|\sigma(t)- s|
$$
for all $s\in U$. We call $f^{\Delta}(t)$ the Hilger (or delta)
derivative of $f$ at $t$.
\end{definition}
For more on the calculus on time scales
we refer the reader to the books \cite{MR1843232,MBP}.


\section{Complex-valued fractional derivatives on time scales}
\label{sec:3}

Let $f:\mathbb{T}\rightarrow \mathbb{R}$ with $\mathbb{T}$
a given time scale. We introduce here a new definition
of fractional (noninteger) delta derivative of order
$\alpha \in (0,1]$ at a point $t\in \mathbb{T}^{\kappa}$.

\begin{definition}[The delta fractional derivative of order $\alpha$]
Assume $f:\mathbb{T}\rightarrow \mathbb{R}$ with
$\mathbb{T}$ a time scale. Let $t\in \mathbb{T}^{\kappa}$
and $\alpha \in (0,1]$. We define $f^{\Delta^{\alpha}}(t)$
to be the number (provided it exists) with the property that given
any $\epsilon >0$ there is a neighborhood $U$ of $t$
(i.e., $U=(t-\delta,t+\delta)\cap \mathbb{T}$
for some $\delta>0$) such that
\begin{equation}
\label{eq:nova:der}
\left|[f^{\alpha}(\sigma(t))-f^{\alpha}(s)]
-f^{\Delta^{\alpha}}(t)[\sigma(t)^{\alpha}-s^{\alpha}]\right|
\leq \epsilon \left|\sigma(t)^{\alpha}-s^{\alpha}\right|
\end{equation}
for all $s\in U$. We call $f^{\Delta^{\alpha}}(t)$
the delta derivative of order $\alpha$ of $f$ at $t$ or
the delta fractional (noninteger order) derivative of $f$ at $t$.
Moreover, we say that $f$ is delta differentiable of order $\alpha$ on
$\mathbb{T}^{\kappa}$ provided $f^{\Delta^{\alpha}}(t)$ exists for all
$t\in\mathbb{T}^{\kappa}$. Function $f^{\Delta^{\alpha}}
:\mathbb{T}^{\kappa}\rightarrow \mathbb{C}$ is then called the
delta derivative of order $\alpha$ of $f$ on $\mathbb{T}^{\kappa}$.
\end{definition}

\begin{remark}
In (\ref{eq:nova:der}) we use $f^{\alpha}$ to denote the power $\alpha$ of $f$.
It is clear that the new derivative coincides with the standard Hilger
derivative in the integer order case $\alpha = 1$. Differently from $\alpha = 1$,
in general $f^{\Delta^{\alpha}}(t)$ is a complex number.
\end{remark}

\begin{theorem}
\label{th1}
Assume $f:\mathbb{T}\rightarrow \mathbb{R}$ with $\mathbb{T}$
a time scale. Let $t\in \mathbb{T}^{\kappa}$ and $\alpha\in\mathbb{R}$.
Then the following proprieties hold:
\begin{enumerate}
\item\label{2} If $f$ is continuous at $t$ and $t$ is right-scattered,
then $f$ is delta differentiable of order $\alpha$ at $t$ with
\begin{equation}
\label{eq1}
f^{\Delta^{\alpha}}(t)
=\frac{f^{\alpha}(\sigma(t))-f^{\alpha}(t)}{\sigma^{\alpha}(t)-t^{\alpha}}.
\end{equation}

\item \label{3} If $t$ is right-dense, then $f$ is delta differentiable
of order $\alpha$ at $t$ if only if the limit
$$
\lim_{s\rightarrow t}\frac{f^{\alpha}(t)-f^{\alpha}(s)}{t^{\alpha}-s^{\alpha}}
$$
exists as a finite number. In this case
\begin{equation}
\label{eq2}
f^{\Delta^{\alpha}}(t)
=\lim_{s\rightarrow t}
\frac{f^{\alpha}(t)-f^{\alpha}(s)}{t^{\alpha}-s^{\alpha}}.
\end{equation}

\item If $f$ is delta differentiable of order $\alpha$ at $t$, then
$$
f^{\alpha}(\sigma(t))
=f^{\alpha}(t)+(\sigma(t)^{\alpha}-t^{\alpha})f^{\Delta^{\alpha}}(t).
$$
\end{enumerate}
\end{theorem}

\begin{proof}
\begin{enumerate}
\item  Assume $f$ is continuous at $t$ and $t$ is right scattered.
By continuity,
$$
\lim_{s\rightarrow t}\frac{f^{\alpha}(\sigma(t))
-f^{\alpha}(s)}{\sigma^{\alpha}(t)-s^{\alpha}}
=\frac{f^{\alpha}(\sigma(t))-f^{\alpha}(t)}{\sigma^{\alpha}(t)
-t^{\alpha}}.
$$
Hence, given $\epsilon>0$, there is a neighborhood $U$ of $t$
such that
$$
\left|\frac{f^{\alpha}(\sigma(t))-f^{\alpha}(s)}{\sigma^{\alpha}(t)-s^{\alpha}}
-\frac{f^{\alpha}(\sigma(t))
-f^{\alpha}(t)}{\sigma^{\alpha}(t)-t^{\alpha}}\right|
\leq\epsilon
$$
for all $s\in U.$ It follows that
$$
\left| f^{\alpha}(\sigma(t))-f^{\alpha}(s)-\frac{f^{\alpha}(\sigma(t))
-f^{\alpha}(t)}{\sigma^{\alpha}(t)-t^{\alpha}}
[\sigma^{\alpha}(t)-s^{\alpha}]\right|\leq\epsilon
\left|\sigma^{\alpha}(t)-s^{\alpha}\right|
$$
for all $s\in U$. Hence, we get the desired result \eqref{eq1}.

\item Assume $f$ is differentiable at $t$ and $t$ is right-dense.
Let $\epsilon>0$ be given. Since $f$ is differentiable at $t$,
there is a neighborhood $U$ of $t$ such that
$$
\mid[f^{\alpha}(\sigma(t))-f^{\alpha}(t)] -f^{\Delta^{\alpha}}(t)[
\sigma^{\alpha}(t)-s^{\alpha} ]\mid
\leq\epsilon\mid\sigma^{\alpha}(t)-s^{\alpha} \mid
$$
for all $s\in U$. Since $\sigma(t)=t$, we have that
$$
\mid[f^{\alpha}(\sigma(t))-f^{\alpha}(t)] -f^{\Delta^{\alpha}}(t)[t^{\alpha}
-s^{\alpha}]\mid \leq\epsilon\mid\sigma^{\alpha}(t)-s^{\alpha} \mid
$$
for all $s\in U$. It follows that
$
\left|\frac{f^{\alpha}(t) -f^{\alpha}(s)}{t^{\alpha}
-s^{\alpha}}-f^{\Delta^{\alpha}}(t) \right|
\leq\epsilon
$
for all $s\in U$, $s\neq t$, and we get the desired
equality \eqref{eq2}. Assume
$
\lim_{s\rightarrow t}
\frac{f^{\alpha}(t)-f^{\alpha}(s)}{t^{\alpha}-s^{\alpha}}
$
exists and is equal to $X$ and $\sigma(t)=t$. Let $\epsilon>0$.
Then there is a neighborhood $U$ of $t$ such that
$$
\left|\frac{f^{\alpha}(\sigma(t))
-f^{\alpha}(s)}{t^{\alpha}-s^{\alpha}}-X\right|
\leq \epsilon
$$
for all $s\in U$. Because
$\mid f^{\alpha}(\sigma(t))
-f^{\alpha}(s)-X(t^{\alpha}-s^{\alpha})\mid
\leq \epsilon |t^{\alpha}-s^{\alpha}|$
for all $s\in U$,
$$
f^{\Delta^{\alpha}}(t)=X
=\lim_{s\rightarrow t} \frac{f^{\alpha}(t)
-f^{\alpha}(s)}{t^{\alpha}-s^{\alpha}}.
$$
\item \label{4}
If $\sigma(t)=t$, then $\sigma^{\alpha}(t)-t^{\alpha}=0$ and
$$
f^{\alpha}(\sigma(t))=f^{\alpha}(t)
=f^{\alpha}(t)+(\sigma^{\alpha}(t)
-t^{\alpha})f^{\Delta^{\alpha}}(t).
$$
On the other hand, if $\sigma(t)>t$, then by item \ref{2}
\begin{eqnarray*}
f^{\alpha}(\sigma(t))&=&f^{\alpha}(t)+(\sigma^{\alpha}(t)-t^{\alpha})
\frac{f^{\alpha}(\sigma(t))-f^{\alpha}(t)}{\sigma(t)^{\alpha}-t^{\alpha}}\\
&=& f^{\alpha}(t)+(\sigma^{\alpha}(t)-t^{\alpha})f^{\Delta^{\alpha}}(t)
\end{eqnarray*}
and the proof is complete.
\end{enumerate}
\end{proof}

\begin{example}
If $\mathbb{T}=\mathbb{R}$, then \eqref{eq2}
yields that $f:\mathbb{R}\rightarrow \mathbb{R}$ is delta differentiable
of order $\alpha$ at $t\in \mathbb{R}$  if and only if
$f^{\Delta^\alpha}(t)=\lim_{s\rightarrow t}
\frac{f^{\alpha}(t)-f^{\alpha}(s)}{t^{\alpha}-s^{\alpha}}$ exists,
i.e., if and only if $f$ is fractional differentiable
at $t$. In this case we get the derivative $f^{(\alpha)}$
of \cite{KAK}.
\end{example}

\begin{example}
If $\mathbb{T}=\mathbb{Z}$, then item \ref{2} of Theorem~\ref{th1} yields
that $f:\mathbb{Z}\rightarrow \mathbb{R}$ is delta-differentiable of order
$\alpha$ at $t\in \mathbb{Z}$  with
$$
f^{\Delta^{\alpha}}(t)=\frac{f^{\alpha}(\sigma(t))
-f^{\alpha}(t)}{\sigma^{\alpha}(t)-t^{\alpha}}
=\frac{f^{\alpha}(t+1)-f^{\alpha}(t)}{(t+1)^{\alpha}-t^{\alpha}}.
$$
\end{example}

\begin{example}
\label{example3}
If $f:\mathbb{T}\rightarrow \mathbb{R}$ is defined
by $f(t)\equiv\lambda \in\mathbb{R}$,
then $f^{\Delta^{\alpha}}(t) \equiv 0$. Indeed,
if $t$ is right-scattered, then by item \ref{2} of Theorem~\ref{th1}
$f^{\Delta^{\alpha}}(t)
=\frac{f^{\alpha}(\sigma(t))-f^{\alpha}(t)}{\sigma^{\alpha}(t)-t^{\alpha}}
=\frac{\lambda^{\alpha}-\lambda^{\alpha}}{\sigma^{\alpha}(t)-t^{\alpha}}=0$;
if $t$ is right-dense, then by \eqref{eq2} we get
$f^{\Delta^{\alpha}}(t)=\lim_{s\rightarrow t}
\frac{\lambda^{\alpha}-\lambda^{\alpha}}{t^{\alpha}-s^{\alpha}}=0$.
\end{example}

\begin{example}
\label{ex1:1}
If $f:\mathbb{T}\rightarrow \mathbb{R}$, $t\mapsto t$,
then $f^{\Delta^{\alpha}} \equiv 1$ because
if $\sigma(t)>t$ (i.e., $t$ is right-scattered), then
$f^{\Delta^{\alpha}}(t)=\frac{f^{\alpha}(\sigma(t))
-f^{\alpha}(t)}{\sigma^{\alpha}(t)-t^{\alpha}}
=\frac{\sigma^{\alpha}(t)-t^{\alpha}}{\sigma^{\alpha}(t)-t^{\alpha}}=1$;
if $\sigma(t)=t$ (i.e., $t$ is right-dense), then
$f^{\Delta^{\alpha}}=\lim_{s\rightarrow t}\frac{f^{\alpha}(t)
-f^{\alpha}(s)}{t^{\alpha}-s^{\alpha}}
=\frac{t^{\alpha}-s^{\alpha}}{t^{\alpha}-s^{\alpha}}=1$.
\end{example}

\begin{example}
\label{ex1:2}
Let $g:\mathbb{T}\rightarrow \mathbb{R}$,
$t\mapsto \frac{1}{t}$. We have
$g^{\Delta^{\alpha}}(t)
= - \frac{1}{\left(t \sigma(t)\right)^{\alpha}}$.
Indeed, if $\sigma(t)=t$, then $g^{\Delta^{\alpha}}(t)=-\frac{1}{t^{2\alpha}}$;
if $\sigma(t)>t$, then
$$
g^{\Delta^{\alpha}}(t)
=\frac{g^{\alpha}(\sigma(t))-g^{\alpha}(t)}{\sigma^{\alpha}(t)-t^{\alpha}}
= \frac{\left(\frac{1}{\sigma(t)}\right)^{\alpha}
-\left(\frac{1}{t}\right)^{\alpha}}{\sigma^{\alpha}(t)-t^{\alpha}}
=\frac{\frac{t^{\alpha}-\sigma^{\alpha}(t)}{t^{\alpha}
\sigma^{\alpha}(t)}}{t^{\alpha}
-\sigma^{\alpha}(t)} = - \frac{1}{t^{\alpha}\sigma^{\alpha}(t)}.
$$
\end{example}

\begin{example}
\label{ex1:3}
Let $h:\mathbb{T}\rightarrow \mathbb{R}$, $t\mapsto t^{2}$.
We have $h^{\Delta^{\alpha}}(t)=\sigma^{\alpha}(t)+t^{\alpha}$.
Indeed, if $t$ is right-dense, then $h^{\Delta^{\alpha}}(t)=\lim_{s\rightarrow t}
\frac{t^{2\alpha}-s^{2\alpha}}{t^{\alpha}-s^{\alpha}}
= 2t^{\alpha}$; if $t$ is right-scattered, then
$$
h^{\Delta^{\alpha}}(t)=\frac{h^{\alpha}(\sigma(t))
-h^{\alpha}(t)}{\sigma^{\alpha}(t)-t^{\alpha}}=\frac{\sigma^{2\alpha}(t)
-t^{2\alpha}}{\sigma^{\alpha}(t)-t^{\alpha}}=\sigma^{\alpha}(t)+t^{\alpha}.
$$
\end{example}

\begin{example}
\label{ex7}
Consider the time scale $\mathbb{T}=h\mathbb{Z}$, $h > 0$.
Let $f$ be the function defined by
$f:h\mathbb{Z}\rightarrow \mathbb{R}$, $t\mapsto (t-c)^2$, $c\in\mathbb{R}$.
The fractional derivative of order $\alpha$ of $f$ at $t$ is
\begin{equation*}
\begin{split}
f^{\Delta^{\alpha}}(t)
&=\frac{f^{\alpha}(\sigma(t))-f^{\alpha}(t)}{\sigma^{\alpha}(t)-t^{\alpha}}
=\frac{((\sigma(t)-c)^{2})^{\alpha}
-((t-c)^{2})^{\alpha}}{\sigma^{\alpha}(t)-t^{\alpha}}\\
&= \frac{(t+h-c)^{2\alpha}-(t-c)^{2\alpha}}{(t+h)^{\alpha}-t^{\alpha}}.
\end{split}
\end{equation*}
\end{example}

\begin{remark}
Examples~\ref{ex1:2}, \ref{ex1:3} and \ref{ex7}
show that in general $f^{\Delta^{\alpha}}(t)$ is a complex number
(for instance, choose  $\alpha=\frac{1}{2}$ and $t<0$).
\end{remark}

\begin{theorem}
\label{th2}
Assume $f,g:\mathbb{T}\rightarrow\mathbb{R}$ are continuous
and delta differentiable of order $\alpha$ at $t\in \mathbb{T}^{\kappa}$.
Then the following proprieties hold:
\begin{enumerate}
\item For any constant $\lambda$, function
$\lambda f:\mathbb{T}\rightarrow \mathbb{R}$ is delta differentiable of order
$\alpha$ at $t$ with
$(\lambda f)^{\Delta^{\alpha}}=\lambda^{\alpha}f^{\Delta^{\alpha}}$.

\item The product $fg:\mathbb{T}\rightarrow \mathbb{R}$ is delta differentiable
of order $\alpha$ at $t$ with
$$
(fg)^{\Delta^{\alpha}}(t)=f^{\Delta^{\alpha}}(t)g^{\alpha}(t)
+f^{\alpha}(\sigma(t))g^{\Delta^{\alpha}}(t)
= f^{\Delta^{\alpha}}(t)g^{\alpha}(\sigma(t))
+f^{\alpha}(t)g^{\Delta^{\alpha}}(t).
$$

\item If $f(t)f(\sigma(t))\neq 0$, then $\frac{1}{f}$
is delta differentiable of order $\alpha$ at $t$ with
$$
\left(\frac{1}{f}\right)^{\Delta^{\alpha}}(t)
=\frac{-f^{\Delta^{\alpha}}(t)}{f^{\alpha}(\sigma(t))f^{\alpha}(t)}.
$$

\item If $g(t)g(\sigma(t))\neq 0$, then
$\frac{f}{g}$  is delta differentiable of order $\alpha$ at $t$ with
$$
\left(\frac{f}{g}\right)^{\Delta^{\alpha}}(t)
= \frac{f^{\Delta^{\alpha}}(t)g^{\alpha}(t)
-f^{\alpha}(t)g^{\Delta^{\alpha}}(t)}{g^{\alpha}(\sigma(t))g^{\alpha}(t)}.
$$
\end{enumerate}
\end{theorem}

\begin{proof}
\begin{enumerate}
\item Let $\epsilon\in(0,1)$. Define
$\epsilon^{\ast}=\frac{\epsilon}{|\lambda|^{\alpha}}\in(0,1)$. Then
there exists a neighborhood $U$ of $t$ such that
$|f^{\alpha}(\sigma(t))-f^{\alpha}(s)-f^{\Delta^{\alpha}}(t)(\sigma^{\alpha}(t)
-s^{\alpha})|\leq\epsilon^{\ast} |\sigma^{\alpha}(t)-s^{\alpha}|$
for all $s\in U $. It follows that
\begin{equation*}
\begin{split}
|(\lambda f)^{\alpha}&(\sigma(t))
-(\lambda f)^{\alpha}(s)-\lambda f^{\Delta^{\alpha}}(t)
(\sigma^{\alpha}(t)-s^{\alpha})|\\
&= |\lambda|^{\alpha}\mid f^{\alpha}(\sigma(t))-f^{\alpha}(s)
-f^{\Delta^{\alpha}}(t)(\sigma^{\alpha}(t)-s^{\alpha})\mid  \\
&\leq \epsilon^{\ast}|\lambda|^{\alpha}|\sigma^{\alpha}(t)-s^{\alpha}|
\leq \frac{\epsilon}{|\lambda|^{\alpha}}|\lambda|^{\alpha}|
\sigma^{\alpha}(t)-s^{\alpha}|
= \epsilon|\sigma^{\alpha}(t)-s^{\alpha}|
\end{split}
\end{equation*}
for all $s\in U$. Thus $(\lambda f)^{\Delta^{\alpha}}(t)
=\lambda^{\alpha}f^{\Delta^{\alpha}}(t)$ holds.

\item Let $\epsilon\in (0,1)$. Define $\epsilon^{\ast}
=\epsilon[1+|f^{\alpha}(t)|+|g^{\alpha}(\sigma(t))|
+|g^{\Delta^{\alpha}}(\sigma(t))| ]^{-1}$. Then
$\epsilon^{\ast}\in(0,1)$ and there exist neighborhoods
$U_{1},U_{2}$ and $U_{3}$ of $t$ such that
$$
|f^{\alpha}(\sigma(t))-f^{\alpha}(s)-f^{\Delta^{\alpha}}(t)(\sigma^{\alpha}(t)
-s^{\alpha})|\leq\epsilon^{\ast} |\sigma^{\alpha}(t)-s^{\alpha}|
$$
for all $s\in U_{1}$,
$|g^{\alpha}(\sigma(t))-g^{\alpha}(s)-g^{\Delta^{\alpha}}(t)(\sigma^{\alpha}(t)
-s^{\alpha})|\leq\epsilon^{\ast} |\sigma^{\alpha}(t)-s^{\alpha}|$
for all $s\in U_{2}$ and such as $f$ is continuous. Then
$|f(t)-f(s)|\leq\epsilon^{\ast}$ for all $s\in U_{3}$.
Define $U=U_{1}\cap U_{2}\cap U_{3}$ and let $s\in U$. It follows that
\begin{equation*}
\begin{split}
|&(fg)^{\alpha}(\sigma(t))-(fg)^{\alpha}(s)
-[g^{\Delta^{\alpha}}(t)f^{\alpha}(t)
+ g^{\alpha}(\sigma(t))f^{\Delta^{\alpha}}(t)][\sigma^{\alpha}(t)-s^{\alpha}]|\\
&= |[f^{\alpha}(\sigma(t))-f^{\alpha}(s)
-f^{\Delta^{\alpha}}(t)(\sigma^{\alpha}(t)-s^{\alpha})](g^{\alpha}(\sigma(t)))
+g^{\alpha}(\sigma(t))f^{\alpha}(s) \\
&\quad + g^{\alpha}(\sigma(t))
f^{\Delta^{\alpha}}(t)(\sigma^{\alpha}(t)
-s^{\alpha})-f^{\alpha}(s)g^{\alpha}(s)\\
&\quad -[g^{\Delta^{\alpha}}(t)f^{\alpha}(t)
+g^{\alpha}(\sigma(t))f^{\Delta^{\alpha}}(t)] [\sigma^{\alpha}(t)-s^{\alpha} ]|\\
&= |[f^{\alpha}(\sigma(t))-f^{\alpha}(s)
-f^{\Delta^{\alpha}}(t)(\sigma^{\alpha}(t)-s^{\alpha})](g^{\alpha}(\sigma(t)))\\
&\quad +[g^{\alpha}(\sigma(t))-g^{\alpha}(s)-g^{\Delta^{\alpha}}(t)(\sigma^{\alpha}(t)
-s^{\alpha})](f^{\alpha}(t))\\
&\quad +[g^{\alpha}(\sigma(t))-g^{\alpha}(s)
-g^{\Delta^{\alpha}}(t)(\sigma^{\alpha}(t)-s^{\alpha})](f^{\alpha}(s)
-f^{\alpha}(t))+ f^{\alpha}(s)g^{\alpha}(s)\\
&\quad +g^{\Delta^{\alpha}}(t) f^{\alpha}(s)(\sigma^{\alpha}(t)-s^{\alpha})
+ g^{\alpha}(\sigma(t))f^{\Delta^{\alpha}}(t)(\sigma^{\alpha}(t)-s^{\alpha})
-g^{\alpha}(s)f^{\alpha}(s)\\
&\quad + g^{\Delta^{\alpha}}(t)
f^{\alpha}(s)(\sigma^{\alpha}(t)-s^{\alpha})
+ g^{\alpha}(\sigma(t))f^{\Delta^{\alpha}}(t)(\sigma^{\alpha}(t)-s^{\alpha})
-f^{\alpha}(s)g^{\alpha}(s)\\
&\quad - [g^{\Delta^{\alpha}}(t)f^{\alpha}(t)
+g^{\alpha}(\sigma(t))f^{\Delta^{\alpha}}(t)][\sigma^{\alpha}(t)-s^{\alpha}]|\\
&\leq |f^{\alpha}(\sigma(t))-f^{\alpha}(s)-f^{\Delta^{\alpha}}(t)
(\sigma^{\alpha}(t)-s^{\alpha})| |(g^{\alpha}(\sigma(t)))|\\
&\quad + |g^{\alpha}(\sigma(t))-g^{\alpha}(s)-g^{\Delta^{\alpha}}(t)(\sigma^{\alpha}(t)
-s^{\alpha})||(f^{\alpha}(t))|\\
&\quad +|g^{\alpha}(\sigma(t))-g^{\alpha}(s)
-g^{\Delta^{\alpha}}(t)(\sigma^{\alpha}(t)-s^{\alpha})||f^{\alpha}(s)-f^{\alpha}(t)|\\
&\quad +|g^{\Delta^{\alpha}}(t)||f^{\alpha}(t)-f^{\alpha}(s)||\sigma^{\alpha}(t)-s^{\alpha}|\\
&= \epsilon^{\ast}|(g^{\alpha}(\sigma(t)))||\sigma^{\alpha}(t)-s^{\alpha}|\\
&\quad +\epsilon^{\ast}|(f^{\alpha}(t))||\sigma^{\alpha}(t)-s^{\alpha}|
+\epsilon^{\ast}|\sigma^{\alpha}(t)-s^{\alpha} |\epsilon^{\ast}
+\epsilon^{\ast}|g^{\Delta^{\alpha}}(t)||\sigma^{\alpha}(t)-s^{\alpha}|\\
&\leq \epsilon^{\ast}|\sigma^{\alpha}(t)-s^{\alpha} |(\epsilon^{\ast}
+|(f^{\alpha}(t))|+|g^{\Delta^{\alpha}}(t)|+|g^{\Delta^{\alpha}}(t)|)\\
&\leq \epsilon^{\ast}|\sigma^{\alpha}(t)-s^{\alpha}|
(1+|(f^{\alpha}(t))|+|g^{\Delta^{\alpha}}(t)|+|g^{\Delta^{\alpha}}(t)|)
= \epsilon |\sigma^{\alpha}(t)-s^{\alpha}|.
\end{split}
\end{equation*}
Thus $(fg)^{\Delta^{\alpha}}(t)=f^{\alpha}(t)g^{\Delta^{\alpha}}(t)
+f^{\Delta^{\alpha}}(t)g^{\alpha}(\sigma(t))$ holds at $t$.
The other product rule follows from this last
equality by interchanging functions $f$ and $g$.

\item We use the delta derivative of a constant (Example~\ref{example3}).
Since $\left(f \cdot \frac{1}{f}\right)^{\Delta^{\alpha}}(t)=0$,
it follows from item 2 that
$\left(\frac{1}{f}\right)^{\Delta^{\alpha}}(t)
f^{\alpha}(\sigma(t))+f^{\Delta^{\alpha}}(t)\frac{1}{f^{\alpha}(t)}=0$.
Because we are assuming $f(t)f(\sigma(t))\neq 0$, one has
$\left(\frac{1}{f}\right)^{\Delta^{\alpha}}(t)
=\frac{-f^{\Delta^{\alpha}}(t)}{f^{\alpha}(\sigma(t))f^{\alpha}(t)}$.

\item For the quotient formula we use items 2 and 3 to compute
\begin{equation*}
\begin{split}
\left(\frac{f}{g}\right)^{\Delta^{\alpha}}(t)
&=\left(f \cdot \frac{1}{g}\right)^{\Delta^{\alpha}}(t)
=f^{\alpha}(t)\left(\frac{1}{g}\right)^{\Delta^{\alpha}}(t)
+f^{\Delta^{\alpha}}(t)\frac{1}{g^{\alpha}(\sigma(t))}\\
&=-f^{\alpha}(t)\frac{g^{\Delta^{\alpha}}(t)}{g^{\alpha}(\sigma(t))
g^{\alpha}(t)}+f^{\Delta^{\alpha}}(t)
\frac{1}{g^{\alpha}(\sigma(t))}\\
&=\frac{f^{\Delta^{\alpha}}(t)g^{\alpha}(t)
-f^{\alpha}(t)g^{\Delta^{\alpha}}(t)}{g^{\alpha}(\sigma(t))g^{\alpha}(t)}.
\end{split}
\end{equation*}
This concludes the proof.
\end{enumerate}
\end{proof}

\begin{remark}
The delta derivative of order $\alpha$
of the sum $f+g:\mathbb{T}\rightarrow \mathbb{R}$
does not satisfy the usual property, that is, in general
$(f+g)^{\Delta^{\alpha}}(t)
\neq(f)^{\Delta^{\alpha}}(t)+(g)^{\Delta^{\alpha}}(t)$.
For instance, let $\mathbb{T}$ be an arbitrary time scale and
$f,g$ be functions defined by $f:\mathbb{T}\rightarrow \mathbb{R}$,
$t\mapsto t$, and $g:\mathbb{T}\rightarrow \mathbb{R}$,
$t\mapsto 2t$. One can easily find that
$(f+g)^{\Delta^{\alpha}}(t)=\sqrt{3}
\neq f^{\Delta^{\alpha}}(t)+g^{\Delta^{\alpha}}(t)
=1+\sqrt{2}$.
\end{remark}

\begin{proposition}
\label{prop}
Let $\alpha\in\mathbb{R}$ and $m\in \mathbb{N}$, $m>1$.
For $g$ defined by $g(t)=t^{m}$ we have
\begin{equation}
\label{eq}
g^{\Delta^{\alpha}}(t)
=\sum_{k=0}^{m-1}(t^{\alpha})^{m-k-1}(\sigma^{\alpha})^{k}(t).
\end{equation}
\end{proposition}

\begin{proof}
We prove the formula by induction. If $m=2$, then $g(t)=t^{2}$
and from Example~\ref{ex1:3} we know that
$g^{\Delta^{\alpha}}(t)
=\sum_{k=0}^{1}(t^{\alpha})^{1-k}(\sigma^{\alpha})^{k}(t)
=t^\alpha+\sigma^\alpha(t)$. Now assume
$$
g^{\Delta^{\alpha}}(t)
=\sum_{k=0}^{m-1}(t^{\alpha})^{m-k-1}(\sigma^{\alpha})^{k}(t)
$$
holds for $g(t)=t^{m}$ and let $G(t)=t^{m+1}=t \cdot g(t)$.
We use the product rule of Theorem~\ref{th2} to obtain
\begin{equation*}
\begin{split}
G^{\Delta^{\alpha}}(t)
&=g^{\alpha}(t)+\sigma^{\alpha}(t)g^{\Delta^{\alpha}}(t)
=(t^{\alpha})^{m}+\sigma^{\alpha}(t)\sum_{k=0}^{m-1}
(t^{\alpha})^{m-k-1}(\sigma^{\alpha})^{k}(t)\\
&=(t^{\alpha})^{m}+\sum_{k=0}^{m-1}(t^{\alpha})^{m-k-1}(\sigma^{\alpha})^{k+1}(t)
=(t^{\alpha})^{m}+\sum_{k=1}^{m-1}(t^{\alpha})^{m-k}(\sigma^{\alpha})^{k}(t)\\
&=\sum_{k=0}^{m}(t^{\alpha})^{m-k}(\sigma^{\alpha})^{k}(t).
\end{split}
\end{equation*}
Hence, by mathematical induction, \eqref{eq} holds.
\end{proof}

\begin{example}
Choose $m = 3$ in Proposition~\ref{prop}. Then
$\left(t^3\right)^{\Delta^\alpha} = t^{2\alpha}
+ \left(t \sigma(t)\right)^\alpha + \sigma^{2\alpha}(t)$.
\end{example}

The notion of fractional derivative here introduced
can be easily extended to any arbitrary real order $\alpha$.

\begin{definition}
Let $\alpha > 0$ and $N\in\mathbb{N}_0$ be such that $N<\alpha \leq N+1$.
Then we define $f^{\Delta^{\alpha}}=\left(f^{\Delta^N}\right)^{\Delta^{\alpha-N}}$,
where $f^{\Delta^N}$ is the usual Hilger derivative of order $N$.
\end{definition}


\begin{acknowledgement}
This research is part of first author's Ph.D. project,
which is carried out at Sidi Bel Abbes University, Algeria.
It was initiated while Bayour was visiting the Department
of Mathematics of University of Aveiro, Portugal, 2015.
The hospitality of the host institution and the financial support
of University of Chlef, Algeria, are here gratefully acknowledged.
Torres was supported by Portuguese funds through CIDMA and FCT,
within project UID/MAT/04106/2013.
\end{acknowledgement}



\end{document}